\documentclass{amsart}
\usepackage{amssymb,amsmath,latexsym,times,tikz,hyperref,amsfonts}

\hypersetup{colorlinks=true, linkcolor=blue, citecolor=blue}

\numberwithin{equation}{section}

\theoremstyle{plain}
\newtheorem{theorem}{Theorem}[section]
\newtheorem{corollary}[theorem]{Corollary}
\newtheorem{lemma}[theorem]{Lemma}

\DeclareMathOperator{\cl}{cl}

\title{The free $m$-cone of a matroid and its $\mathcal{G}$-invariant}

\author[J.~Bonin]{Joseph E.~Bonin} \address
{Department of Mathematics\\ The George Washington University\\
  Washington, D.C. 20052, USA} \email {jbonin@gwu.edu, kevinlong@gwmail.gwu.edu}
\author[K.~Long]{Kevin Long} 

\date{\today}

\subjclass{05B35}

\begin{document}
	
\begin{abstract}
  For a matroid $M$, its configuration determines its
  $\mathcal{G}$-invariant. Few examples are known of pairs of matroids
  with the same $\mathcal{G}$-invariant but different configurations.
  In order to produce new examples, we introduce the free $m$-cone
  $Q_m(M)$ of a loopless matroid $M$, where $m$ is a positive integer.
  We show that the $\mathcal{G}$-invariant of $M$ determines the
  $\mathcal{G}$-invariant of $Q_m(M)$, and that the configuration of
  $Q_m(M)$ determines $M$; so if $M$ and $N$ are nonisomorphic and
  have the same $\mathcal{G}$-invariant, then $Q_m(M)$ and $Q_m(N)$
  have the same $\mathcal{G}$-invariant but different configurations.
  We prove analogous results for several variants of the free
  $m$-cone.  We also define a new matroid invariant of $M$, and show
  that it determines the Tutte polynomial of $Q_m(M)$.
\end{abstract}

\maketitle

\section{Introduction}\label{section:introduction}

Of the following three matroid invariants, listed from weakest to
strongest, we are mainly interested in the second and third (we recall
them below and treat them more fully in Section
\ref{section:background}):
\begin{enumerate}
\item[(1)] the Tutte polynomial,
\item[(2)] the $\mathcal{G}$-invariant (introduced by Derksen
  \cite{G-inv}), and
\item[(3)] the configuration (introduced by Eberhardt \cite{config}).
\end{enumerate}
Having two matroid invariants, one of which can be derived from the
other, raises a basic question: how can we construct matroids that
share the weaker invariant but not the stronger?  We address this
question for the $\mathcal{G}$-invariant and configuration. Eberhardt
proved that the configuration of a matroid $M$ determines its Tutte
polynomial. Bonin and Kung \cite{catdata} strengthened this result,
proving that the configuration of $M$ determines its
$\mathcal{G}$-invariant. We treat a construction with which we can
produce matroids with the same $\mathcal{G}$-invariant but different
configurations.

The natural starting point is the most well known matroid invariant,
the Tutte polynomial.  Given a matroid $M=(E,r)$, its \textit{Tutte
  polynomial} is
$$T(M;x,y)=\sum_{A\subseteq E}(x-1)^{r(E)-r(A)}(y-1)^{|A|-r(A)}.$$
The data in this polynomial is the multiset
$\{(|A|,r(A))\,:\,A\subseteq E\}$.  The significance of the Tutte
polynomial comes in part from its being a universal invariant for the
deletion-contraction rule \cite{Tutte}.

Derksen \cite{G-inv} introduced the $\mathcal{G}$-invariant, which
generalizes the Tutte polynomial.  Derksen and Fink \cite{valuative}
later showed that the $\mathcal{G}$-invariant is a universal valuative
invariant for subdivisions of matroid base polytopes.  For a matroid
$M=(E,r)$ where $|E|=n$, the data in its $\mathcal{G}$-invariant is
the multiset of $n$-tuples of rank increases
$$\bigl(r(\{e_1\}),r(\{e_1,e_2\}) - r(\{e_1\}),
r(\{e_1,e_2,e_3\}) - r(\{e_1,e_2\}),\ldots,r(E)-r(E-\{e_n\})\bigr)$$
over all permutations $e_1,e_2,\ldots,e_n$ of $E$.  Bonin and Kung
\cite{catdata} showed that this data is equivalent to the multiset of
$(r(E)+1)$-tuples
$$(|F_0|,|F_1-F_0|,|F_2-F_1|,\ldots,|F_{r(E)}-F_{r(E)-1}|)$$
over all flags, that is, chains $(F_0,F_1,\ldots,F_{r(E)})$ of flats
of $M$ where $r(F_i)=i$.

Eberhardt \cite{config} showed that, for a matroid with no coloops,
the Tutte polynomial can be computed from a small amount of data about
the cyclic flats (the flats that are unions of circuits), namely, from
the abstract lattice formed by the cyclic flats, along with the size
and rank of the cyclic flat corresponding to each element in this
lattice.  This data is the configuration of the matroid. Eberhardt's
result was extended in Bonin and Kung \cite{catdata}, who showed that
the configuration of a matroid with no coloops determines its
$\mathcal{G}$-invariant.  (Extending this to all matroids by also
recording the number of coloops is routine.)

We develop a construction that yields examples that show that the
configuration is strictly stronger than the $\mathcal{G}$-invariant:
we show how to construct pairs of matroids with the same
$\mathcal{G}$-invariant and different configurations. (When we started
this work, the only other such examples we knew of were those treated
in \cite{catdata}, namely, for $k\geq 4$, rank-$k$ Dowling matroids
based on nonisomorphic groups of the same order.  Since then, a very
different technique for constructing such examples has been developed;
see \cite{diffconfig}.)  Let $M$ be a loopless matroid and $m$ be a
positive integer. We will define a matroid $Q_m(M)$ that we call the
free $m$-cone of $M$. Our main result is that if $M$ and $N$ are
nonisomorphic and have the same $\mathcal{G}$-invariant, then $Q_m(M)$
and $Q_m(N)$ have the same $\mathcal{G}$-invariant and different
configurations.

We use the notation and terminology in Oxley \cite{oxley}.  More
specialized background is treated in Section \ref{section:background}.
In Section \ref{section:free cone}, we define the free $m$-cone and
develop some of its properties, which we then use to prove that the
$\mathcal{G}$-invariant of $M$ determines that of its free $m$-cone,
and that the configuration of the free $m$-cone of $M$ determines
$M$. In Section \ref{section:variants}, we treat several variants of
the free $m$-cone and show that, with some exceptions for small $m$,
our main results in Section \ref{section:free cone} also hold for
them.  The Higgs lift of a matroid is a special case of these
variants.  In Section \ref{section:size rank coloop}, we introduce a
matroid invariant that lies between the Tutte polynomial and the
$\mathcal{G}$-invariant, and show that it determines the Tutte
polynomial of the free $m$-cone and of variants of the free $m$-cone.

We use $[n]$ to denote the set $\{1,2,\ldots,n\}$.

\section{Background}\label{section:background}

Given a rank-$k$ matroid $M$ on $E=[n]$ with rank
function $r$ and a permutation $\pi$ on $E$, the \textit{rank
  sequence} $\underline{r}(\pi)=r_1r_2\ldots r_n$ is given by
$r_1=r(\pi(1))$ and, for $i>1$,
$$r_i=r(\{\pi(j)\,:\, j\in [i] \})
-r(\{\pi(j)\,:\, j\in[i-1] \}).$$ Thus, $\{\pi(i)\,: \, r_i=1\}$ is a
basis of $M$.  Each rank sequence is an $(n,k)$\textit{-sequence},
that is, a sequence of $k$ ones and $n-k$ zeroes.  For each
$(n,k)$-sequence $\underline{r}$, let $[\underline{r}]$ be a formal
symbol, and let $\mathcal{G}(n,k)$ be the vector space over a field of
characteristic zero consisting of all formal linear combinations of
such symbols. The $\mathcal{G}$-\textit{invariant of} $M$ is defined
by $$\mathcal{G}(M)=\sum_\pi [\underline{r}(\pi)]$$ where the sum is
over all permutations $\pi$ of $E$.

Another perspective on $\mathcal{G}(M)$ was developed in
\cite{catdata}.  An $(n,k)$\emph{-composition} is an integer sequence
$(a_i)=(a_0, a_1, \ldots, a_k)$ for which $a_0+a_1+\cdots +a_k=n$
where $a_0\geq 0$ and $a_i>0$ for $i\in [k]$.  Let the matroid $M$ be
as above.  A \emph{flag} of $M$ is a sequence
$(X_i)=(X_0, X_1, \ldots, X_k)$ where $X_i$ is a rank-$i$ flat of $M$,
and $X_i\subset X_{i+1}$ for $i<k$.  The \emph{composition} of a flag
$(X_i)$ is $(a_i)$ where $a_0=|X_0|$ and $a_i=|X_i-X_{i-1}|$ for
$i\in [k]$. Thus, $(a_i)$ is an $(n,k)$-composition. Let
$\nu(M;(a_i))$ be the number of flags of $M$ with composition
$(a_i)$. The \emph{catenary data} of $M$ is the
$\binom{n}{k}$-dimensional vector $(\nu(M;(a_i)))$ indexed by
$(n,k)$-compositions.

Bonin and Kung \cite{catdata} defined a special basis of
$\mathcal{G}(n,k)$, the $\gamma$-basis, whose vectors $\gamma((a_i))$
are indexed by $(n,k)$-compositions $(a_i)$.  The change-of-basis
result from \cite{catdata}, stated next, connects $\mathcal{G}(M)$ and
the catenary data of $M$.

\begin{theorem}\label{thm:catdata}
  The catenary data of $M$ determines $\mathcal{G}(M)$ and conversely
  since
  $$\mathcal{G}(M)=\sum_{(a_i)}\nu(M;(a_i))\gamma((a_i)).$$
\end{theorem}

Given a matroid $M=(E, r)$, a subset $A$ of $E$ is \textit{cyclic} if
$M|A$ has no coloops. Just as lines and planes refer to flats of ranks
$2$ and $3$, respectively, if such flats are cyclic, then we call them
cyclic lines or cyclic planes.  More generally, a cyclic flat is a flat
that is cyclic.  The set $\mathcal{Z}(M)$ of cyclic flats of $M$ is a
lattice under inclusion.  The configuration, introduced in Eberhart
\cite{config}, is the abstract lattice of cyclic flats of $M$, without
the sets but with their size and rank. More precisely, the
\textit{configuration} of a matroid $M$ with no coloops is the triple
$(L, s, \rho)$, where $L$ is a lattice and $s$ and $\rho$ are
functions with domain $L$ where there is an isomorphism
$\phi: L\to \mathcal{Z}(M)$ for which $s(x)=|\phi(x)|$ and
$\rho(x)=r(\phi(x))$ for all $x\in L$.  Many nonisomorphic matroids
can have the same configuration: O.~Giménez constructed $n!$
non-paving matroids of rank $2n+ 2$ on $4n+ 5$ elements, all with the
same configuration \cite[see Theorem 5.7]{cyclic flats}.

Bonin and Kung \cite{catdata} showed that if $M$ has no coloops, then
$\mathcal{G}(M)$ can be found from the configuration of $M$.  If $M$
has coloops, $\mathcal{G}(M)$ can be found from
$\mathcal{G}(M\setminus X)$ and $|X|$ where $X$ is the set of coloops,
so we focus on matroids without coloops.

\begin{figure}[t]
  \centering
  \begin{tikzpicture}[scale=1.1]
    \draw[thick](0,0)--(1.5,0);
    \draw[thick](0,-.75)--(1.5,-.75);
    \filldraw (0,0) node {} circle  (2.2pt);
    \node (111) at (0,.23) {1};
    \filldraw (.75,0) node {} circle  (2.2pt);
    \node at (.75,.23) {2};
    \filldraw (1.5,0) node {} circle  (2.2pt);
    \node at (1.5,.23) {3};
    \filldraw (0,-.75) node {} circle  (2.2pt);
    \node at (0,-.52) {4};
    \filldraw (.75,-.75) node {} circle  (2.2pt);
    \node at (.75,-.52) {5};
    \filldraw (1.5,-.75) node {} circle  (2.2pt);
    \node at (1.5,-.52) {6};
    \node at (1,-1.3) {$M_1$};
  \end{tikzpicture}
  \hspace{30pt}
  \begin{tikzpicture}[scale=1.1]
    \draw[thick](0,0)--(2,.5);    
    \draw[thick](0,0)--(2,-.5);    
    \filldraw (0,0) node {} circle  (2.2pt);    
    \node at (0,.23) {1};
    \filldraw (1,.25) node {} circle  (2.2pt);    
    \node at (1,.48) {2};    
    \filldraw (2,.5) node {} circle  (2.2pt);    
    \node at (2,.73) {3};    
    \filldraw (1,-.25) node {} circle  (2.2pt);    
    \node at (1,-0.5) {4};    
    \filldraw (2,-.5) node {} circle  (2.2pt);    
    \node at (2,-.75) {5};    
    \filldraw (1.8,0) node {} circle  (2.2pt);    
    \node at (2.05,0) {6};    
    \node at (1,-1) {$M_2$};
\end{tikzpicture}
\caption{Two rank-$3$ sparse paving matroids.}
\label{fig:continuing examples}
\end{figure}
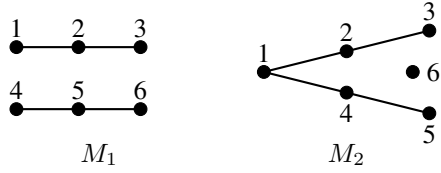

In Figures \ref{fig:continuing examples} and \ref{fig:configuration
  examples}, we give an example of nonisomorphic matroids $M_1$ and
$M_2$ that have the same configuration.  (We will refer to this
example throughout the paper.)  It follows that $M_1$ and $M_2$ have
the same $\mathcal{G}$-invariant, which is
$$\mathcal{G}(M_1)=\mathcal{G}(M_2)=648[111000]+72[110100].$$
They then have the same catenary data:
$$\nu(M_1;0,1,2,3)=  \nu(M_2;0,1,2,3)
= 6 \text{ \ and \ } \nu(M_1;0,1,1,4)= \nu(M_2;0,1,1,4) = 18.$$

  \begin{figure}[t]
    \centering
    \begin{tikzpicture}[scale = 1]
      \node (111) at (1, 3) {$\{1,2,3,4,5,6\}$};
      \node (110) at (0, 2) {$\{1,2,3\}$};
      \node (101) at (2, 2) {$\{4,5,6\}$};
      \node (100) at (1, 1) {$\varnothing$};
      \draw[very thick] (111) -- (110);
      \draw[very thick] (111) -- (101);
      \draw[very thick] (110) -- (100);
      \draw[very thick] (101) -- (100);      
      \node at (1,0.25) {(a) \ $\mathcal{Z}(M_1)$};
    \end{tikzpicture}
    \hspace{20pt} 
    \begin{tikzpicture}[scale = 1]
      \node (111) at (1, 3) {$\{1,2,3,4,5,6\}$};
      \node (110) at (0, 2) {$\{1,2,3\}$};
      \node (101) at (2, 2) {$\{1,4,5\}$};
      \node (100) at (1, 1) {$\varnothing$};			
      \draw[very thick] (111) -- (110);
      \draw[very thick] (111) -- (101);
      \draw[very thick] (110) -- (100);
      \draw[very thick] (101) -- (100);
      \node at (1,0.25) {(b) \ $\mathcal{Z}(M_2)$};
    \end{tikzpicture}
	    \hspace{20pt} 
    \begin{tikzpicture}[scale = 1]
      \node (111) at (1, 3) {$(6,3)$};
      \node (110) at (0, 2) {$(3,2)$};
      \node (101) at (2, 2) {$(3,2)$};
      \node (100) at (1, 1) {$(0,0)$};
      \draw[very thick] (111) -- (110);
      \draw[very thick] (111) -- (101);
      \draw[very thick] (110) -- (100);
      \draw[very thick] (101) -- (100);
      \node at (1,0.25) {(c)};
    \end{tikzpicture}
    \caption{Parts (a) and (b) show the lattice of cyclic flats of
      $M_1$ and $M_2$. Replacing each set by just its size and rank
      gives the configuration of $M_1$ and $M_2$, shown in part (c).}
\label{fig:configuration examples}
\end{figure}
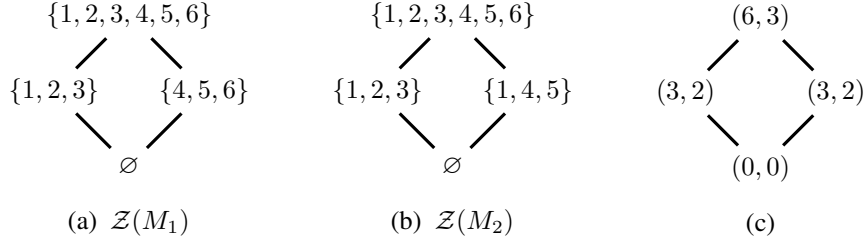

Brylawski \cite{affine} showed that a matroid is determined by its
cyclic flats and their ranks, which is more data than the
configuration gives.  We will use the following related result of Sims
\cite{sims} and Bonin and de Mier \cite{cyclic flats}.

\begin{theorem}\label{thm:cfaxioms}
  For a set $\mathcal{Z}$ of subsets of a set $E$ and a function
  $r : \mathcal{Z}\to \mathbb{Z}$, there is a matroid $M$ on $E$ with
  $\mathcal{Z}(M) = \mathcal{Z}$ and $r_M(X) = r(X)$ for all
  $X\in\mathcal{Z}$ if and only if
  \begin{itemize}
  \item[\emph{(Z0)}] $(\mathcal{Z},\subseteq)$ is a lattice,
  \item[\emph{(Z1)}] $r(0_\mathcal{Z})=0$, where $0_\mathcal{Z}$ is
    the least set in $\mathcal{Z}$,
  \item[\emph{(Z2)}] $0<r(Y)-r(X)<|Y-X|$ for all $X,Y \in \mathcal{Z}$
    with $X\subsetneq Y$, and
  \item[\emph{(Z3)}] for all sets $X, Y$ in $\mathcal{Z}$ (or,
    equivalently, just incomparable sets in $\mathcal{Z}$),
    \begin{equation}\label{ineq:submcyc}
      r(X\vee Y)+r(X\wedge Y)+|(X\cap Y)-(X\wedge Y)|\leq r(X)+r(Y).
    \end{equation}
  \end{itemize}
\end{theorem}

\section{The Free $m$-Cone}\label{section:free cone}

Let $M=(E,r)$ be a rank-$k$ loopless matroid.  For each integer
$m\geq 1$, we will define a rank-$(k+1)$ matroid $Q_m(M)$, which we
often shorten to $Q$, on a set of $(m+1)|E|+1$ elements. The
construction is illustrated in Figure \ref{fig:free cones in detail}.
For each $e\in E$, let $T_e$ be a set of size $m$ that is disjoint
from $E$ and from all other sets $T_{e'}$, and let $T$ be the union of
these $|E|$ sets.  Let $a$ be an element not in $E\cup T$, which we
call the \emph{tip} of $Q$.  Let $E(Q)=E\cup T\cup \{a\}$.

Define $q:2^{E}\to 2^{E(Q)}$ as follows: for $S\subseteq E$,
$$q(S)=S\cup \{a\}\cup\Bigl(\bigcup_{e\in S}T_e\Bigr).$$
Thus, $q(\varnothing)=\{a\}$.  In the examples in Figure \ref{fig:free
  cones in detail}, if $e\in[6]$, then $q(\{e\})$ is the $3$-point
line $\{e,\overline{e},a\}$.  Define $p:2^{E(Q)}\to 2^{E}$ as follows:
for $S\subseteq E(Q)$,
$$p(S)=(S\cap E)\cup\{e:  S\cap T_e\ne \varnothing \}.$$
For $e\in E$ and $x\in T_e$, we have $p(\{e\})=\{e\}= p(\{x\})$.
Henceforth we omit the braces when applying $p$ and $q$ to singleton
sets.  The set $q(S)$ will be the ground set of the free $m$-cone of
$M|S$, so, for convenience, we refer to $q(S)$ as the \emph{cone of
  $S$}.

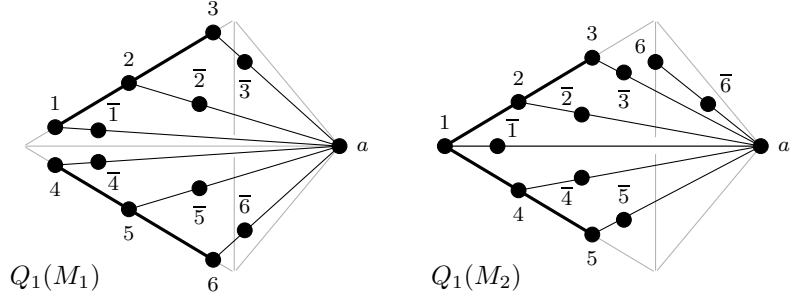
\begin{figure}[t]
  \centering
\begin{tikzpicture}[scale=1.4]

  \node[inner sep = 0] at (0, 1)   (a) {};
  \node[inner sep = 0] at (2, -0.2)   (b) {};
  \node[inner sep = 0] at (2, 2.2)   (c) {};
  \node[inner sep = 0] at (3, 1)   (d) {};

  \draw[black!30](b)--(c);
  \filldraw[white] (1.9,0.9) rectangle (2.1,1.1);
  \draw[black!30](a)--(b)--(d)--(c)--(a)--(d);

  \draw[very thick,black](0.3,1.18)--(1.8,2.08);
  \draw[very thick,black](0.3,0.82)--(1.8,-0.08);

   \draw[black](3,1)--(0.3,1.18);
   \draw[black](3,1)--(1,1.6);
   \draw[black](3,1)--(1.8,2.08);

   \draw[black](3,1)--(0.3,0.82);
   \draw[black](3,1)--(1,0.4);
   \draw[black](3,1)--(1.8,-0.08);  

 \filldraw (3,1) node[right=3pt] {\footnotesize $a$} circle  (2pt);

 \filldraw (0.3,1.18) node[above=3pt] {\footnotesize $1$} circle  (2pt);
 \filldraw (0.71,1.15) node[above right] {\footnotesize $\overline{1}$} circle  (2pt);

 \filldraw (1,1.6) node[above=3pt] {\footnotesize $2$} circle (2pt);
 \filldraw (1.67,1.4) node[above=3pt] {\footnotesize $\overline{2}$} circle (2pt);

 \filldraw (1.8,2.08) node[above=3pt] {\footnotesize $3$} circle  (2pt);
 \filldraw (2.1,1.8) node[below=3pt] {\footnotesize $\overline{3}$} circle (2pt);

 \filldraw (0.3,0.82) node[below=3pt] {\footnotesize $4$} circle  (2pt);
 \filldraw (0.71,0.85) node[below right] {\footnotesize $\overline{4}$} circle (2pt);

 \filldraw (1,0.4) node[below=3pt] {\footnotesize $5$} circle (2pt);
 \filldraw (1.67,0.6) node[below=3pt] {\footnotesize $\overline{5}$} circle (2pt);

 \filldraw (1.8,-0.08) node[below=3pt] {\footnotesize $6$} circle  (2pt);
 \filldraw (2.1,0.2) node[above =3pt] {\footnotesize $\overline{6}$} circle  (2pt);

      \node at (0.3,-0.25) {$Q_1(M_1)$};
 

  \node[inner sep = 0] at (4, 1)   (aa) {};
  \node[inner sep = 0] at (6, -0.2)   (bb) {};
  \node[inner sep = 0] at (6, 2.2)   (cc) {};
  \node[inner sep = 0] at (7, 1)   (dd) {};

  \draw[black!30](bb)--(cc);
  \filldraw[white] (5.93,0.93) rectangle (6.08,1.08);
  \draw[black!30](aa)--(bb)--(dd)--(cc)--(aa)--(dd);

  \draw[very thick,black](4,1)--(5.4,1.84);
  \draw[very thick,black](4,1)--(5.4,0.16);

   \draw[black](7,1)--(4,1);
   \draw[black](7,1)--(4.7,1.42);
   \draw[black](7,1)--(5.4,1.84);

   \draw[black](7,1)--(5.4,0.16);
   \draw[black](7,1)--(4.7,0.58);
   \draw[black](7,1)--(6,1.8);  

  \filldraw (7,1) node[right=3pt] {\footnotesize $a$} circle  (2pt);

  \filldraw (4,1) node[above=3pt] {\footnotesize $1$} circle  (2pt);
  \filldraw (4.5,1) node[above right] {\footnotesize $\overline{1}$} circle  (2pt);

  \filldraw (4.7,1.42) node[above=3pt] {\footnotesize $2$} circle (2pt);
  \filldraw (5.3,1.3) node[above left] {\footnotesize $\overline{2}$} circle (2pt);

 \filldraw (5.4,1.84) node[above=3pt] {\footnotesize $3$} circle  (2pt);
 \filldraw (5.7,1.7) node[below=3pt] {\footnotesize $\overline{3}$} circle (2pt);

  \filldraw (4.7,0.58) node[below=3pt] {\footnotesize $4$} circle  (2pt);
  \filldraw (5.3,0.7) node[below left] {\footnotesize $\overline{4}$} circle (2pt);

  \filldraw (5.4,0.16) node[below=3pt] {\footnotesize $5$} circle (2pt);
  \filldraw (5.7,0.3) node[above=3pt] {\footnotesize $\overline{5}$} circle (2pt);

  \filldraw (6,1.8) node[above left] {\footnotesize $6$} circle  (2pt);
  \filldraw (6.5,1.4) node[above right =1pt] {\footnotesize $\overline{6}$} circle  (2pt);

        \node at (4.3,-0.25) {$Q_1(M_2)$};
 
\end{tikzpicture}
\caption{The free $1$-cones of the matroids $M_1$ and $M_2$, on
  $E=[6]$, in Figure \ref{fig:continuing examples}.  Each of $M_1$ and
  $M_2$ is shown in a face of a tetrahderon.  The tip $a$ is at the
  vertex opposite that face.  For each $e\in [6]$, the set $T_e$ is
  $\{\overline{e}\}$.  }\label{fig:free cones in detail}
\end{figure}

Let
$\mathcal{Z}(Q) = \mathcal{Z}(M)\cup\{q(F)\,:\, F\text{ is a nonempty
  flat of }M\}$.  In each example in Figure \ref{fig:free cones in
  detail}, besides the cyclic flats of the original matroid
($\varnothing$, two $3$-point lines, and $E$), this set contains the
six $3$-point lines that contain $a$, the cone of each $2$- and
$3$-point line of the original matroid (e.g.,
$\{3,\overline{3},6,\overline{6},a\}$ and
$\{1,\overline{1},2,\overline{2}, 3,\overline{3},a\}$), and $q(E)$.
Define $r_Q:\mathcal{Z}(Q)\to\mathbb{Z}$ by
\begin{equation}\label{eq:conerank}
  r_Q(Z) =
  \begin{cases}
    r(Z), & \text{ if } Z\in\mathcal{Z}(M)\\
    r(F)+1, & \text{ if } Z=q(F)\text{ where } F \text{ is a nonempty
      flat of } M.
  \end{cases}
\end{equation}	
We next show that $\mathcal{Z}(Q)$ and $r_Q$ define a matroid on
$E(Q)$.  We denote this matroid by $Q_m(M)$ and call it the \emph{free
  $m$-cone} of $M$.
	
\begin{theorem}
  Let $M=(E, r)$ be a loopless matroid.  The set $\mathcal{Z}(Q)$ and
  the function $r_Q$ defined above satisfy axioms \emph{(Z0)--(Z4)} in
  Theorem \emph{\ref{thm:cfaxioms}} and so define a matroid on $E(Q)$.
\end{theorem}

\begin{proof}
  Since $\varnothing\in \mathcal{Z}(M)$, the least set in
  $\mathcal{Z}(Q)$ is $\varnothing$, so property (Z1) holds by
  Equation (\ref{eq:conerank}).  Each of properties (Z0), (Z2), and
  (Z3) involves two sets in $\mathcal{Z}(Q)$.  When both sets are in
  $\mathcal{Z}(M)$, the properties hold since $M$ is a matroid and
  $r_Q(X)=r(X)$ for such sets, so we focus on the case having at least
  one set not in $\mathcal{Z}(M)$.  We use $\vee$ and $\wedge$ for the
  join and meet operations of $\mathcal{Z}(Q)$, and $\vee_M$ and
  $\wedge_M$ for those of $\mathcal{Z}(M)$. We use $\cl$ for the
  closure operator of $M$.
  
  Let $X$ and $Y$ be distinct sets in $\mathcal{Z}(Q)-\mathcal{Z}(M)$,
  so $X=q(X')$ and $Y=q(Y')$ where $X'$ and $Y'$ are nonempty flats of
  $M$.  Now $X\cap Y = q(X'\cap Y')$, which is in $\mathcal{Z}(Q)$
  unless the flat $X'\cap Y'$ of $M$ is empty.  So
  $X\wedge Y= q(X'\cap Y')$ unless $X'\cap Y'=\varnothing$, in which
  case $X\wedge Y= \varnothing$.  Since $X\vee Y$ can only be the cone
  of a flat of $M$, and the least flat of $M$ that contains $X'$ and
  $Y'$ is $\cl(X'\cup Y')$, we have $X\vee Y=q(\cl(X'\cup Y'))$.  For
  property (Z3), submodularity in $M$ gives
  $r(\cl(X'\cup Y'))+r(X'\cap Y')\leq r(X')+r(Y')$.  If
  $X'\cap Y'\ne\varnothing$, then $(X\cap Y)-(X\wedge Y)=\varnothing$,
  so Inequality (\ref{ineq:submcyc}) follows by Equation
  (\ref{eq:conerank}).  If $X'\cap Y'=\varnothing$, then
  $X\cap Y =\{a\}$, $X'\wedge Y'=\varnothing$, and
  $r(\cl(X'\cup Y'))\leq r(X')+r(Y')$, and Inequality
  (\ref{ineq:submcyc}) follows in this case too by Equation
  (\ref{eq:conerank}).  For property (Z2), we assume, in addition,
  that $X\subsetneq Y$.  Thus, $X'\subsetneq Y'$.  Equation
  (\ref{eq:conerank}) gives $r_Q(Y)-r_Q(X)=r(Y')-r(X')>0$. Since
  $|Y-X|=(m+1)|Y'-X'|$, property (Z2) follows.
 
  Now assume that $X\in \mathcal{Z}(M)$ and
  $Y \in \mathcal{Z}(Q)-\mathcal{Z}(M)$; let $Y=q(Y')$ where $Y'$ is a
  nonempty flat of $M$.  Now $X\cap Y=X\cap Y'$, which, as an
  intersection of flats of $M$, is a flat of $M$.  Let $C$ be the set
  of coloops of $M|X\cap Y$.  Now $(X\cap Y)-C$ is a cyclic flat of
  $M$, and it contains each cyclic flat that is a subset of $X$ and
  $Y$, so it is $X\wedge Y$.  Also,
  $$r_Q(X\wedge Y)=r(X\cap Y)-|C|=r(X\cap Y)-|(X\cap Y)-(X\wedge
  Y)|.$$ Since $X\vee Y$ can only be the cone of a flat of $M$, and
  the least flat of $M$ that contains $X$ and $Y'$ is $\cl(X\cup Y')$,
  we have $X\vee Y=q(\cl(X\cup Y'))$.  This completes the proof of
  property (Z0).  To finish the proof of property (Z3), Inequality
  (\ref{ineq:submcyc}) follows since
  \begin{align*}
    \,r_Q(X\vee Y)+r_Q(X\wedge Y)+|(X\cap Y) -(X\wedge Y)|= 
    &\, r(X\cup Y')+1+r(X\cap Y')\\
    \leq&\, r(X)+r(Y')+1\\
    =&\,r_Q(X)+r_Q(Y).
  \end{align*}
  For property (Z2), assume, in addition, that $X\subsetneq Y$.
  Property (Z2) holds if $Y=q(X)$ since
  $r_Q(q(X))-r_Q(X)=1<1+m|X|=|q(X)-X|$.  If $Y\ne q(X)$, then property
  (Z2) follows from this and the inequality
  $0<r_Q(Y)-r_Q(q(X))<|Y-q(X)|$ proven in the previous paragraph.\qed
\end{proof}

This result justifies extending $r_Q$, as defined in equation
(\ref{eq:conerank}), so that $r_Q$ is the rank function of the matroid
$Q_m(M)$.

If $M$ is binary, graphic, or transversal, $Q_m(M)$ might not share
these properties.  However, if $M$ is representable over $\mathbb{R}$,
then so is $Q_m(M)$, and if $M$ is representable over some field of
characteristic $p$, then so is $Q_m(M)$, although not necessarily over
the same field of characteristic $p$.  This holds since another way to
define $Q_m(M)$ (less suited to our purposes) is via iterated
principal extensions of the direct sum of $M$ and the rank-$1$ matroid
on the tip, adding points freely to the lines $\cl_M(\{a,e\})$ for
$e\in E(M)$.

Note that in the examples in Figure \ref{fig:free cones in detail},
and their generalization for any $m\geq 1$, each cyclic line in
$Q_m(M_1)$ is in a cyclic plane that contains four cyclic lines, but
that fails for $Q_m(M_2)$, so the configurations differ.  Our results
below imply that $\mathcal{G}(Q_m(M_1))$ and $\mathcal{G}(Q_m(M_2))$
are equal.  So, if $m\geq 1$, then $Q_m(M_1)$ and $Q_m(M_2)$ have the
same $\mathcal{G}$-invariant and different configurations.  Our main
theorem generalizes this example.

\begin{theorem}\label{thm:main result}
  Let $M$ and $N$ be nonisomorphic loopless matroids with
  $\mathcal{G}(M)=\mathcal{G}(N)$.  For all $m\geq 1$, $Q_m(M)$ and
  $Q_m(N)$ have the same $\mathcal{G}$-invariant and different
  configurations.
\end{theorem}

This result can be iterated, producing $Q_{m_1}(Q_m(M))$ and
$Q_{m_1}(Q_m(N))$, and so on.

We prove Theorem \ref{thm:main result} in two parts: in Theorem
\ref{thm:config}, we show that the configuration of $Q_m(M)$
determines $M$; in Theorem \ref{thm:G-inv}, we show that
$\mathcal{G}(M)$ determines $\mathcal{G}(Q_m(M))$.  We first treat
some preliminary results that enter into the proofs of both theorems.
By Theorem \ref{thm:catdata}, $\mathcal{G}(M)$ is equivalent to the
catenary data of $M$. We will show that the catenary data of $M$
determines that of $Q_m(M)$. To do so, we use the next few results to
characterize the flats of $Q_m(M)$, which then allows us to
characterize its flags.

\begin{lemma}\label{lemma:coloops}
  Fix $m\geq 1$.  Let $M=(E,r)$ be a loopless matroid and let $Q=Q_m(M)$.
  \begin{itemize}
  \item[\emph{(i)}] If $F$ is a flat of $Q$ and $a\notin F$, then all
    elements of $F-E$ are coloops of $Q|F$.
  \item[\emph{(ii)}] All flats of $M$ are flats of $Q$ and the
    restriction $Q|E$ is $M$.
  \end{itemize}
\end{lemma}

\begin{proof}
  For any circuit $C$ with $C\not\subseteq E$, its closure $\cl_Q(C)$
  is a cyclic flat and $\cl_Q(C)\not\subseteq E$, so $a\in\cl_Q(C)$.
  So for any flat $F$ of $Q$ with $a\notin F$, no circuit of $Q|F$
  contains elements of $F-E$, so statement (i) follows.
  
  If $e\in E$ and $x\in T_e$, then
  $a\in q(e)\subseteq \cl_Q(\{x,e\})$, so the only flat of $Q$ that
  properly contains $E$ is $E(Q)$.  Let $Y$ be the (perhaps empty) set
  of coloops of $M$, so $E-Y$ is in $\mathcal{Z}(M)$ and so in
  $\mathcal{Z}(Q)$.  Thus, $r(E)=r(E-Y)+|Y|=r_Q(E-Y)+|Y|\geq r_Q(E)$.
  Since $\cl_Q(E\cup a)=E(Q)$ and $Q$ has rank $r(E)+1$, it follows
  that $r(E)=r_Q(E)$ and that $E$ is a flat of $Q$.  If $F$ is a flat
  of $M$, then $F=q(F)\cap E$ and, as an intersection of flats of $Q$,
  it is a flat of $Q$, so the first part of statement (ii) holds.  By
  considering chains of flats, we now get $r(F)=r_Q(F)$ for all flats
  $F$ of $M$.
  
  If the second part failed, then there would be a flat $X$ of $Q|E$
  of minimal rank, say $i$, that is not a flat of $M$.  Now $i>0$.
  Let $Y$ be a flat of rank $i-1$ with $Y\subsetneq X$. Fix $x$ in
  $X-Y$.  Only one flat of rank $i$ in $Q$ contains $Y\cup x$, and
  $\cl_M(Y\cup x)$ has this property, so $X=\cl_M(Y\cup x)$, contrary
  to $X$ not being a flat of $M$.  This completes the proof.\qed
\end{proof}

Having $r(X)=r_Q(X)$ for $X\subseteq E$ allows us to simplify the
notation by using $r$ for the rank function of $Q$.

The next lemma is useful for describing the flats of $Q$.  While we
will apply it to the free $m$-cone of a matroid, the lemma holds for
any matroid.

\begin{lemma}\label{lemma:lines}
  Let $L_1,L_2,\ldots,L_j$ be distinct lines in a matroid $N$, each
  containing the rank-$1$ flat $\{a\}$.  Fix sets
  $X=\{x_1,x_2,\ldots,x_j\}$ and $Y=\{y_1,y_2,\ldots,y_j\}$ where
  $x_i,y_i \in L_i-a$ for each $i\in [j]$.  Then $X\cup a$ is
  independent if and only if $Y\cup a$ is independent.  Thus, if $F$
  is a flat of $N$ with $a\not\in F$ and $X\cup Y\subseteq F$, then
  $X$ is independent if and only if $Y$ is independent.
\end{lemma}

\begin{proof}
  Let $F=\cl(L_1\cup L_2\cup \cdots \cup L_j)$.  If $X\cup a$ is
  independent, then it is a basis of $N|F$.  Since
  $F\subseteq \cl(Y\cup a)$ and $|X\cup a|=|Y\cup a|$, it follows that
  $Y\cup a$ is also a basis of $N|F$, and so is independent.  The
  result follows by symmetry.\qed
\end{proof}

\begin{corollary}\label{cor:rank}
  Let $Q=Q_m(M)=(E(Q), r)$ for some $m\geq 1$. For $S\subseteq E(Q)$,
  \begin{equation}\label{eqn:rankofprojection}
    r(p(S)) =
    \begin{cases}
      r(S), & \text{if } a\not\in \cl(S),\\
      r(S)-1, & \text{if } a\in \cl(S).
    \end{cases}
  \end{equation}
\end{corollary}

\begin{proof}
  Assume that $a\notin\cl(S)$.  No two elements of $S$ are colinear
  with $a$, so by Lemma \ref{lemma:lines}, a subset $U$ of $S$ is
  independent if and only if $p(U)$ is independent.  Therefore
  $r(S)=r(p(S))$.  Now assume that $a\in\cl(S)$.  Let $B$ be a basis
  of $Q|S\cup a$ with $a\in B$.  So $p(B-a)$ is independent by Lemma
  \ref{lemma:lines}, so $r(p(S))\geq r(S)-1$.  Now $r(p(S))<r(S)$
  since $p(S)\subseteq \cl(S)$ and $a\in\cl(S)- \cl(p(S))$.  Thus,
  $r(p(S))= r(S)-1$.\qed
\end{proof}

The next lemma characterizes the flats of $Q$.  Recall that
$T=E(Q)-(E\cup a)$ is the set of elements of $Q$ that are neither the
tip, $a$, nor in $M$.

 \begin{lemma}\label{lemma:flats}
   Let $Q=Q_m(M)$ for some $m\geq 1$. The flats $F$ of $Q$ that
   contain $a$ are the cones of flats in $M$; indeed, $F=q(F\cap
   E)$. If $a\not\in F$, then $F$ is a flat of $Q$ if and only if
   \begin{enumerate}
   \item[\emph{(1)}] $F\cap E$ is a flat of $M$,
   \item[\emph{(2)}] $p(x)\neq p(y)$ for all $x, y\in F$ with
     $x\neq y$, and
   \item[\emph{(3)}] for each basis $B$ of $M|F\cap E$, the set
     $B\cup p(F\cap T)$ is independent.
   \end{enumerate}
 \end{lemma}

 \begin{proof}
   If $F$ is a flat of $Q$ with $a\in F$, then $F\cap E$ (an
   intersection of flats) is a flat of $Q$ and of $M$, and $F$ is its
   cone.  The converse holds by the definition of $\mathcal{Z}(Q)$.
  
   Let $F$ be a flat of $Q$ not containing $a$. Condition (1) holds
   since $F\cap E$ (an intersection of flats of $Q$) is a flat of $Q$,
   and so of $M$.  If $p(x)=p(y)$ for some $x,y\in F$ with $x\neq y$,
   then $a\in\cl(\{x,y\})\subseteq F$, which is a contradiction, so
   condition (2) holds.  Let $B$ be a basis of $M|F\cap E$. By Lemma
   \ref{lemma:coloops}, since $a\notin F$, elements in $F\cap T$ are
   coloops of $Q|F$, so $B\cup (F\cap T)$ is independent.  Therefore
   $B\cup p(F\cap T)$ is independent by Lemma \ref{lemma:lines}, so
   condition (3) holds.

   For the converse, take $F\subseteq E(Q)$ with $a\not\in F$
   satisfying conditions (1)--(3).  Assume that $F$ is not a flat of
   $Q$.  Then there is a circuit $C$ of $Q$ and element $x\in C-F$
   with $C-x\subseteq F$.  By condition (1), $F\cap E$ is a flat of
   $M$, and hence of $Q$, so $C-x\not\subseteq E$.  Now $\cl(C)$ is a
   cyclic flat and $\cl(C)\not\subseteq E$, so $a\in \cl(C)$, and so
   $a\in\cl(C-x)$.  Corollary \ref{cor:rank} now gives
   $r(p(C-x)) = |C-x|-1$.  This contradicts condition (3) since
   condition (2) gives $|p(C-x)|=|C-x|$.\qed
\end{proof}

\begin{corollary}\label{cor:flags}
  Fix $m\geq 1$.  Let $M$ be a rank-$k$ loopless matroid and let
  $Q=Q_m(M)$. Let $(Y_i)=(Y_0,Y_1,\ldots,Y_{k+1})$ be a flag in
  $Q$. Fix $j\in [k+1]$. If $a\in Y_{j-1}$, then $Y_{j-1}=q(X_{j-2})$
  and $Y_j=q(X_{j-1})$ where $X_{j-2}$ and $X_{j-1}$ are flats of $M$
  of rank $j-2$ and $j-1$ respectively with $X_{j-2}\subset X_{j-1}$.
  If $a\notin Y_{j-1}$, then $Y_j$ satisfies exactly one of the
  following conditions:
  \begin{itemize}
  \item[\emph{(i)}] $Y_j=\cl(Y_{j-1}\cup a)=q(\cl(p(Y_{j-1})))$,
  \item[\emph{(ii)}] $Y_j-E=Y_{j-1}-E$ and
    $r(Y_j\cap E)=r(Y_{j-1}\cap E)+1$,
  \item[\emph{(iii)}] $Y_j=Y_{j-1}\cup x$, where $x\in T$ and
    $p(x)\notin \cl(p(Y_{j-1}))$.
  \end{itemize}
  Case \emph{(i)} occurs when $Y_j$ is the first flat in $(Y_i)$ that
  contains $a$.
\end{corollary}

\begin{proof}
  Corollary \ref{cor:rank} and Lemma \ref{lemma:flats} give the result
  when $a\in Y_{j-1}$.  Now assume that $a\notin Y_{j-1}$.  Then
  $r(Y_{j-1})=r(p(Y_{j-1}))$ by Corollary \ref{cor:rank}.  Option (i)
  accounts for the unique rank-$j$ flat that contains $Y_{j-1}\cup a$.
  Now assume $a\not\in Y_j$.  If some $x \in Y_j-Y_{j-1}$ is in $T$,
  then $x$ is a coloop in $Q|Y_j$ by Lemma \ref{lemma:coloops}, so
  $Y_j=Y_{j-1}\cup x$ and $p(x)\notin \cl(p(Y_{j-1}))$.  Finally, if
  $Y_j-Y_{j-1}\subseteq E$, then $Y_j-E = Y_{j-1}-E$.\qed
\end{proof}

We now prove the first result into which we divide the proof of
Theorem \ref{thm:main result}.

\begin{theorem}\label{thm:config}
  Fix $m\geq 1$. Let $M=(E,r)$ be a loopless matroid and $Q$ be its
  free $m$-cone. The configuration of $Q$ determines $M$.
\end{theorem}

\begin{proof}
  The cases with $r(M)\leq 2$ are easy, so assume that $r(M)>2$.  Let
  $\mathcal{F}_1$ be the set of rank-$1$ flats of $M$.  Consider sets
  $\mathcal{C}$ of lines in $\mathcal{Z}(Q)$ that satisfy
  \begin{itemize}
  \item[(L1)] for each $L\in\mathcal{C}$, at most one proper, nonempty
    subset of $L$ is in $\mathcal{Z}(Q)$,
  \item[(L2)]
    $\bigvee\limits_{L\in \mathcal{C}}L=E(Q)$, and
  \item[(L3)] if $L_1, L_2\in \mathcal{C}$ with $L_1\neq L_2$, then
    $r(L_1\vee L_2)=3$.
  \end{itemize}
	
  By Lemma \ref{lemma:flats}, each cyclic line of $Q$ is either a cone
  $q(F)$ with $F\in \mathcal{F}_1$ (and so contains $a$) or a cyclic
  line of $M$. Let $\mathcal{L}$ be the set of all lines of $Q$ that
  contain $a$.  From the definition of $Q$, it follows that
  $\mathcal{L}$ satisfies properties (L1)--(L3).  If $L$ is a line
  with $L\subseteq E$, then there is an $F\in \mathcal{F}_1$ disjoint
  from $L$ (since $r(M)>2$) and $r(L\vee q(F))>3$.  So $\mathcal{L}$
  is a maximal set with these properties.
  
  Let $\mathcal{C}$ be a set of lines satisfying properties (L1)--(L3)
  with $\mathcal{C}\not\subseteq\mathcal{L}$.  Fix $L\in \mathcal{C}$
  with $L\subseteq E$.  By property (L2), $\mathcal{C}$ must contain a
  line $J$ with $a\in J$.  We claim that $J$ is the only such line.
  To see this, assume that $J'\in\mathcal{C}-\{J\}$ and $a\in J'$. Let
  $J\cap E=P$ and $J'\cap E=P'$.  Any line of $M$ that is coplanar
  with both $J$ and $J'$ must contain $P\cup P'$.  Exactly one line of
  $M$ contains $P\cup P'$, so $L$ is the only line of $M$ that is
  coplanar with both $J$ and $J'$, and so all lines of $\mathcal{C}$
  other than $L$ contain $a$.  However, the only lines containing $a$
  that $L$ is coplanar with are $q(\cl(x))$ for $x\in L$, so
  $\mathcal{C}$ would fail property (L2).  So $L'\subseteq E$ for each
  $L'\in \mathcal{C}-\{J\}$.
   
  Each $F\in\mathcal{F}_1$ is contained in a unique line in
  $\mathcal{L}$, so $|\mathcal{L}|= |\mathcal{F}_1|$.  Since all lines
  in $\mathcal{C}$ are cyclic, by property (L1) each line in
  $\mathcal{C}$ contains at least three rank-$1$ flats.  By property
  (L3), if two lines of $\mathcal{C}$, say $L_1$ and $L_2$, contain
  some $F\in\mathcal{F}_1$, then all lines in $\mathcal{C}$ contain
  $F$ since for any such line $L$, the planes $\cl(L_1\cup L)$ and
  $\cl(L_2\cup L)$ intersect in the line $L$ and contain the point
  $F$.  If $F\in\mathcal{F}_1$ is in all lines in $\mathcal{C}$, then
  $|\mathcal{C}|\leq \frac{|\mathcal{F}_1|-1}{2}+1$, otherwise
  $|\mathcal{C}|\leq \frac{|\mathcal{F}_1|}{3}+1$ (the $+1$ is for
  $J$).  So $|\mathcal{L}|>|\mathcal{C}|$.  So $\mathcal{L}$ is the
  unique set of largest size satisfying properties (L1)--(L3).

  We can detect properties (L1)--(L3) from the configuration, so we
  can identify $\mathcal{L}$ in the configuration of $Q$.  The
  configuration does not give the sets in $\mathcal{Z}(Q)$, but for
  each $X,Y\in\mathcal{Z}(Q)$, it gives us $|X|$, $|Y|$, and whether
  $X$ and $Y$ are comparable.  To get $M$ up to isomorphism, first
  pick pairwise disjoint sets $X_L$, one for each $L\in\mathcal{L}$,
  where if there is a (necessarily unique) nonempty cyclic flat $Y$
  with $Y< L$, then $|X_L|=|Y|$, otherwise $|X_L|=1$.  The cyclic
  flats in $\mathcal{Z}(Q)-\mathcal{Z}(M)$ are the cones of nonempty
  flats of $M$, so the flats of $M$, up to isomorphism, are, for each
  $F\in\mathcal{Z}(Q)$, the union of the sets $X_L$ such that
  $L \in \mathcal{L}$ and $L\leq F$ (note that $\varnothing$ is such a
  union).\qed
\end{proof}

To complete the proof of Theorem \ref{thm:main result}, we work mainly
with the catenary data, which, by Theorem \ref{thm:catdata}, is
equivalent to the $\mathcal{G}$-invariant.  We first develop one of
the keys to the proof: for a size-$n$, rank-$k$ loopless matroid $M$,
we define a bijection from a certain set of $4$-tuples, one component
of which is a flag of $M$, onto the set of flags of its free $m$-cone
$Q$.
 
To motivate the bijection, which appears in the lemma below, we first
consider how the flags of $Q$ relate to those of $M$.  If
$(Y_i)=(Y_0,Y_1,\ldots,Y_{k+1})$ is a flag of $Q$, then there is an
integer $h$ with $0\leq h\leq k$ so that $a\in Y_i$ if and only if
$i>h$.  By Lemma \ref{lemma:flats}, each $Y_i$ with $i>h$ is the cone
$q(X_{i-1})$ of a flat $X_{i-1}$ of $M$; the flats $X_h,\ldots,X_k$
are the end of a flag of $M$.  Let $|Y_h-E|=b$.  The elements of
$Y_h-E$ are coloops of $Q|Y_h$ by Lemma \ref{lemma:coloops}, so the
list $Y_0\cap E,Y_1\cap E, \ldots, Y_h\cap E$ contains $h-b+1$
distinct flats of $M$, say $X_0,X_1,\ldots,X_{h-b}$; these flats begin
a flag of $M$.  To complete $X_0,X_1,\ldots,X_{h-b}$ and
$X_h,\ldots,X_k$ to a flag of $M$, consider the elements of $Y_h-E$,
say $e_1,e_2,\ldots,e_b$, listed in the order in which each first
appears in the flag $(Y_i)$.  For each $j\in[b]$, let
$X_{h-b+j}=\cl\bigl(X_{h-b}\cup p(\{e_1, e_2, \ldots,
e_j\})\bigr)$. From Corollary \ref{cor:rank}, it follows that
$(X_i)=(X_0, X_1,\ldots, X_k)$ is a flag of $M$.  Each $e_j$ has
$p(e_j)\in X_{h-b+j}-X_{h-b+j-1}$.

\begin{lemma}\label{lemma:bijection}
  Fix $m\geq 1$.  Let $M$ be a size-$n$, rank-$k$ loopless matroid and
  let $Q$ be $Q_m(M)$.   Let $\mathcal{T}$ be the set of all $4$-tuples
  $\bigl((X_i),h,C,(e_j)\bigr)$ where
  \begin{itemize}
  \item[$\bullet$] $(X_i)$ is a flag of $M$,
  \item[$\bullet$] $h$ is an integer with $0\leq h\leq k$,
  \item[$\bullet$] $C$ is a subset of $[h]$, and
  \item[$\bullet$] $(e_i)=(e_1, e_2,\ldots, e_{|C|})$ is a $|C|$-tuple
    of elements of $T$ where, for $i$ with $1\leq i\leq |C|$, we have
    $p(e_i)\in X_{h-|C|+i}- X_{h-|C|+i-1}$.
  \end{itemize}
  Write the set $C$ in a $4$-tuple in $\mathcal{T}$ as
  $\{c_1,c_2,\ldots,c_{|C|}\}$ with $c_1<c_2<\cdots<c_{|C|}$, and the
  difference $[h]-C$ as $D=\{d_1,d_2,\ldots,d_{h-|C|}\}$ with
  $d_1<d_2<\cdots<d_{h-|C|}$.  The $j$th element of $C$ refers to
  $c_j$, and likewise for $D$.  Let $\mathcal{F}$ be the set of flags
  of $Q$.  Define $f:\mathcal{T}\to\mathcal{F}$ by
  $f\bigl((X_i),h,C,(e_j)\bigr)=(Y_i)$ where
  \begin{equation*} Y_i =
    \begin{cases} \varnothing, & \text{if } i=0,\\
      Y_{i-1}\cup e_j, & \text{if } i \text{ is the $j$th
        element of } C,\\  
      Y_{i-1}\cup X_j, & \text{if } i \text{ is the $j$th element of }
      D,\\ 
      q(X_{i-1}), & \text{if }
      h<i\leq k+1.
    \end{cases}
  \end{equation*}
  Then the map $f$ is a bijection.
\end{lemma}
  
\begin{proof}
  We first show that $(Y_i)$ is indeed a flag of $Q$.  Clearly
  $Y_i\subsetneq Y_{i+1}$, so we focus on showing that each $Y_i$ is a
  flat of $Q$ and $r(Y_i)=i$.  This is clear for $Y_0$ as well as for
  $Y_i$ with $h<i\leq k+1$.  For $0<i\leq h$, we show that $Y_i$
  satisfies the conditions in Lemma \ref{lemma:flats}, and so is a
  flat. Condition (1) is immediate.  Condition (2) holds since (i)
  $Y_h-E=\{e_1,e_2,\ldots,e_{|C|}\}$, and (ii) $p(e_j)$ is in
  $X_{h-|C|+j}- X_{h-|C|+j-1}$, and all $|C|$ such differences are
  disjoint from $X_{h-|C|}$ and each other.  For condition (3), the
  intersection $Y_i\cap E$ is $X_l$ for some $l$ with
  $0\leq l\leq h-|C|$, and each element $e_j$ in $Y_i$ has
  $p(e_j)\in X_{h-|C|+j}-X_{h-|C|+j-1}$, so for any basis $B$ of
  $M|Y_i\cap E$, we get $B\cup p(Y_i\cap T)$ by adjoining to $B$ a
  sequence of elements, none of which is in the closure of $B$ and the
  earlier elements in the sequence.  Thus, $B\cup p(Y_i\cap T)$ is
  independent, so condition (3) holds.  So $Y_i$ is a flat.  If $i$ is
  the $j$-th element of $C$, then $e_j$ is a coloop of $Y_i$; if $i$
  is the $j$-th element of $D$, then $Y_i\cap T=Y_{i-1}\cap T$, and
  $Y_i\cap E=X_j$ and $Y_{i-1}\cap E=X_{j-1}$; both options give
  $r(Y_i)=r(Y_{i-1})+1$, and so we get $r(Y_i)=i$ by induction.

  The paragraph before the proof in effect treats the inverse: map a
  flag $(Y_i)$ of $Q$ to the $4$-tuple $\bigl((X_i),h,C,(e_j)\bigr)$
  where $(X_i)$, $h$, and $(e_j)$ are as stated in that paragraph and
  $C=\{i\,:\,Y_i-Y_{i-1}=\{e_j\} \text{ for some } j\}$.  So $f$ is a
  bijection.
\end{proof}

  We next prove the result that completes the proof of Theorem
\ref{thm:main result}.

\begin{theorem}\label{thm:G-inv}
  Fix $m\geq 1$.  Let $M$ be a size-$n$, rank-$k$ loopless matroid and
  let $Q$ be $Q_m(M)$. Then $\mathcal{G}(Q)$ can be determined from
  $\mathcal{G}(M)$.
\end{theorem}

\begin{proof}
  As stated above, we work mainly with the catenary data, and for that
  we use the bijection $f$ in Lemma \ref{lemma:bijection}.  More
  precisely, from the composition of the flag $(X_i)$ of $M$ and the
  other three entries in a $4$-tuple in
  $\bigl((X_i),h,C,(e_j)\bigr)\in\mathcal{T}$, we find the composition
  of the image $f\bigl((X_i),h,C,(e_j)\bigr)$.  With this, we express
  $\mathcal{G}(Q)$ using the catenary data of $M$.

  By Theorem \ref{thm:catdata},
  $$\mathcal{G}(Q)=\sum_{(b_i)}\nu(Q;(b_i))\gamma((b_i))$$ where the
  sum is over all $((m+1)n+1,k+1)$-compositions
  $(b_i)=(b_0, b_1, \ldots, b_{k+1})$, and $\gamma((b_i))$ is the
  $\gamma$-basis vector indexed by $(b_i)$. Let
  $f ((X_i), h, C, (e_j))=(Y_i)$, let $(a_i)$ be the composition of
  $(X_i)$, and let $(b_i)$ be that of $(Y_i)$.  We next show that
  $(b_i)$ depends only on $(a_i)$, $h$, and $C$; this justifies
  letting $\hat{f}((a_i), h, C)$ denote $(b_i)$.  Since $Q$ has no
  loops, $b_0=0$. For $i\in [h]$, if $i\in C$, then $Y_i-Y_{i-1}$ is a
  singleton subset of $T$, so $b_i=1$; if $i$ is the $j$-th element of
  $D$, then $Y_i-Y_{i-1}=X_j- X_{j-1}$, so $b_i=a_j$.  If $i>h$, then
  $Y_i=q(X_{i-1})$, so $b_{h+1}=1+\sum_{j=1}^h(a_j(m+1)-b_j)$, and
  $b_i=(m+1) a_{i-1}$ for $i>h+1$.

  The triple $((a_i), h, C)$ determines another key item: there are
  $m \cdot a_{h-|C|+j} $ options for $e_j$ since
  $e_j \in \bigl(q(X_{h-|C|+j})-q(X_{h-|C|+j-1})\bigr)-E$.
  With this, we get
  \begin{equation*}
    \mathcal{G}(Q) =\sum_{(a_i)}\nu(M;(a_i))\sum\limits_{h=0}^k\sum_{C\subseteq
      [h]}\left(\prod\limits_{j=1}^{|C|}
      m \cdot a_{h-|C|+j}\right)\gamma(\hat{f}((a_i), h, C)). 
  \end{equation*}
  The sums account for all triples $((a_i), h, C)$.  The term
  $\nu(M;a_i)$ accounts for all flags $(X_i)$ of $M$ that have
  composition $(a_i)$.  When $h=0$, we take $[h]=\varnothing$.  The
  product accounts for the choices of $(e_1, e_2,\ldots, e_{|C|})$,
  and $\gamma(\hat{f}((a_i), h, C))$ is the $\gamma$-basis element
  that all flags of $Q$ of the form $f((X_i), h, C,(e_j))$, where
  $(X_i)$ and $(e_j)$ vary over all of their respective options,
  contribute to $\mathcal{G}(Q)$.\qed
\end{proof} 

The next corollary better fits our main result.

\begin{corollary}
  If $M$ and $N$ are loopless and $\mathcal{G}(M)=\mathcal{G}(N)$,
  then $\mathcal{G}(Q_m(M))=\mathcal{G}(Q_m(N))$.
\end{corollary}

\section{Variants of the free $m$-cone}\label{section:variants}

As we show in this section, the proof of Theorem \ref{thm:G-inv}
adapts to give analogous results for the following deletions of the
free $m$-cone: the \textit{tipless $m$-cone} $Q\setminus a$, the
\textit{baseless $m$-cone} $Q\setminus E$, and the \textit{tipless
  baseless $m$-cone} $Q\setminus (E\cup a)$.

We first prove a result noted by Kung \cite{perscomm}: these variants
include an important, well-studied construction, the Higgs lift.
Recall that the \emph{Higgs lift} of a matroid $M=(E,r)$ is the
matroid $L(M)=(E,r_L)$ where $r_L(X) = \min\{r(X)+1,|X|\}$ for all
$X\subseteq E$.

\begin{theorem}\label{thm:higgs}
  If $M$ has no loops, then the tipless baseless $1$-cone
  $Q\setminus(E\cup a)$ of $M$ is isomorphic to the Higgs lift
  $L(M)$ of $M$.
\end{theorem}

\begin{proof}
  For a $1$-cone, the restriction $p:2^T\to 2^E$ of the map $p$
  defined earlier is a bijection, and $p$ gives a bijection from the
  ground set of $Q\setminus(E\cup a)$ onto that of $L(M)$.  We use
  $p^{-1}$ to relabel $L(M)$ so that its ground set is $T$.  With this
  relabeling, the rank function $r_L$ of $L(M)$ is given by
   \begin{equation*}\label{eqn:rankofprojection}
    r_L(S) =
    \begin{cases}
      r(p(S)), & \text{if }  p(S) \text{ is independent in } M, \\
      r(p(S))+1, & \text{otherwise,}
    \end{cases}
  \end{equation*}
  for $S\subseteq T$.  By Corollary \ref{cor:rank}, the rank of $S$ in
  $Q\setminus(E\cup a)$ is given by
   \begin{equation*}\label{eqn:rankofprojection}
    r(S) =
    \begin{cases}
      r(p(S)), & \text{if } a\not\in \cl(S),\\
      r(p(S))+1, & \text{otherwise.}
    \end{cases}
  \end{equation*}
  When $p(S)$ is independent, Lemma \ref{lemma:lines} applies to the
  sets $S$ and $p(S)$; now $p(S)\cup a$ is independent since
  $a\not\in \cl(p(S))$, so $S\cup a$ is independent, so
  $a\not\in\cl(S)$; therefore, $r_L(S)=r(p(S))=r(S)$.  When $p(S)$ is
  dependent in $M$, either (i) $p(S)$ contains parallel elements, and
  so $S$ contains two points that are collinear with $a$, so
  $a\in \cl(S)$, or (ii) Lemma \ref{lemma:lines} applies and gives
  that $S$ is dependent, but the closure of any circuit in $T$
  contains $a$, so $a\in \cl(S)$; thus, $r_L(S)=r(p(S))+1=r(S)$.\qed
\end{proof}

Given a matroid $M$, a flag $(X_i)$ of $M$, and $S\subset E(M)$, we
say that $(X_i)$ \textit{collapses} in $M\setminus S$ if
$X_i-S=X_{i+1}- S$ for some $i$.  If it is obvious what $S$ is, we
just say that $(X_i)$ collapses. The next lemma relates the flags of
$M$ to those of $M\setminus S$ when at least one flag of $M$ does not
collapse.

\begin{lemma}\label{lemma:flag deletion}
  Let $S$ be a nonempty subset of $E(M)$.  Not all flags of $M$
  collapse in $M\setminus S$ if and only if $r(M)=r(M\setminus S)$.
  In that case, the flags of $M\setminus S$ are exactly the sequences
  $(Y_i-S)$ where $(Y_i)$ is a flag of $M$ that does not collapse.
\end{lemma}

\begin{proof}
  Let $(Y_i)$ be a flag of $M$ that does not collapse. The $r(M)+1$
  flats in the chain of flats $(Y_i-S)$ in $M\setminus S$ strictly
  increase in rank, so $r(M\setminus S)=r(M)$ and $(Y_i-S)$ is a flag
  of $M\setminus S$. Now assume that $r(M\setminus S)=r(M)$ and let
  $(Z_i)$ be a flag of $M\setminus S$.  Since $r(\cl_M(Z_i))=i$ and
  $r(M\setminus S)=r(M)$, it follows that $(\cl_M(Z_i))$ is a flag of
  $M$ and $(\cl_M(Z_i)-S)=(Z_i)$.\qed
\end{proof}

\begin{theorem}\label{thm:tipless}
  Fix $m\geq 1$.  Let $M$ be a rank-$k$ loopless matroid and let
  $Q=Q_m(M)$.  The $\mathcal{G}$-invariants of $Q\setminus a$,
  $Q\setminus E$, and $Q\setminus (E\cup a)$ can be computed from
  $\mathcal{G}(M)$.
\end{theorem}

\begin{proof}
  A flag $(Y_i)$ of $Q$ collapses in $Q\setminus a$ if and only if
  $Y_{i+1}-Y_i=\{a\}$ for some $i$, which, by Corollary
  \ref{cor:flags}, occurs if and only if $Y_1=\{a\}$.  So, using the
  bijection $f$ in Lemma \ref{lemma:bijection}, a flag $(Y_i)$ of $Q$
  does not collapse if and only if $(Y_i)=f((X_i), h, C, (e_j))$ for
  some $((X_i), h, C, (e_j))\in \mathcal{T}$ with $ h\geq 1$.  The
  composition $(b_i)$ of the flag $(Y_i-a)$ of $Q\setminus a$ is the
  same as that of the flag $(Y_i)$ except that
  $b_{h+1}=\sum_{j=1}^h(a_j(m+1)-b_j)$. Letting $f_1((a_i), h, C)$
  denote $(b_i)$, we get
  \begin{equation*}
    \mathcal{G}(Q\setminus a) =\sum_{(a_i)}\nu(M;(a_i))\sum
    \limits_{h=1}^k\sum_{C\subseteq
      [h]}\left(\prod\limits_{j=1}^{|C|}
      m \cdot   a_{h-|C|+j}\right)\gamma(f_1((a_i), h,C)). 
  \end{equation*}

  A flag $(Y_i)$ of $Q$ collapses in $Q\setminus E$ if and only if
  $Y_{i+1}-Y_i\subseteq E$ for some $i$, which, by Corollary
  \ref{cor:flags}, occurs if and only if
  $r(Y_{i+1}\cap E)=r(Y_i\cap E)+1$ for some $i$ such that
  $a\notin Y_{i+1}$.  So, a flag $(Y_i)$ of $Q$ does not collapse if
  and only if $(Y_i)=f((X_i), h, C, (e_j))$ for some
  $((X_i), h, C, (e_j))\in \mathcal{T}$ with $C=[h]$.  The composition
  $(b_i)$ of the flag $(Y_i-E)$ is the same as that of $(Y_i)$, except
  that $b_{h+1}=1+\sum_{j=1}^h (m a_j-1)$ and $b_i=m a_{i-1}$ for
  $i>h+1$. Letting $f_2((a_i),h)$ denote $(b_i)$, since $C=[h]$, we get
  \begin{equation*}
    \mathcal{G}(Q\setminus E)
    =\sum_{(a_i)}\nu(M;(a_i))\sum\limits_{h=0}^k
    \left(\prod\limits_{j=1}^h m \cdot  a_j\right)\gamma(f_2((a_i), h)).
  \end{equation*}

  For $Q\setminus(E\cup a )$, we treat the cases $m=1$ and $m>1$
  separately, with the latter first.  A flag $(Y_i)$ of $Q$ collapses
  in $Q\setminus (E\cup a)$ if and only if
  $Y_{i+1}-Y_i\subseteq E\cup a$ for some $i$.  By Corollary
  \ref{cor:flags}, the flag $(Y_i)$ does not collapse if and only if
  $Y_1\neq\{a\}$, and for all $Y_{i+1}$ with $a\notin Y_{i+1}$, we
  have $Y_{i+1}-Y_i=\{x\}$ for some $x\in T$.  So, a flag $(Y_i)$ of
  $Q$ does not collapse if and only if $(Y_i)=f((X_i), h, C, (e_j))$
  for some $((X_i), h, C, (e_j))\in \mathcal{T}$ with $C=[h]$ and
  $h\geq 1$.  The composition $(b_i)$ of the flag $(Y_i-(E\cup a))$
  of $Q\setminus (E\cup a )$ is the same as that of $(Y_i)$, except
  that $b_{h+1}=\sum_{j=1}^h (ma_j-1)$ and $b_i=ma_{i-1}$ for
  $i>h+1$. Letting $f_3((a_i),h)$ denote $(b_i)$, since $C=[h]$, we get
  \begin{equation*}
    \mathcal{G}(Q\setminus (E\cup a)) =\sum_{(a_i)}\nu(M;(a_i)
    \sum\limits_{h=1}^k\left(\prod\limits_{j=1}^h
      m \cdot a_j\right)\gamma(f_3((a_i), h)).
  \end{equation*}

  Now consider $m=1$, which has one more way for the inclusion
  $Y_{i+1}-Y_i\subseteq E\cup a$ to occur, namely, if $i=h$,
  $Y_h\subseteq T$, and $p(Y_h)$ is a (necessarily independent) flat
  of $M$, since then $Y_{h+1}=\cl(Y_h\cup a)=Y_h\cup p(Y_h)\cup a$.
  So a flag $(Y_i)$ of $Q$ does not collapse if and only if
  $(Y_i)=f((X_i), h, C, (e_j))$ for some
  $((X_i), h, C, (e_j))\in \mathcal{T}$ with $C=[h]$, $h\geq 1$, and
  $|X_i-X_{i-1}|>1$ for at least one $i\leq h$.  When such flags
  exist, the composition $(b_i)$ of the flag $(Y_i-(E\cup a))$ of
  $Q\setminus (E\cup a)$ is the same as when $m>1$. The formula for
  the $\mathcal{G}$-invariant is also the same as when $m>1$, except
  that for a composition $(a_i)$ and integer $h \in [k]$ to
  contribute, we must have some $i\in [h]$ with $a_i>1$.

  Finally, if all flags of $Q$ collapse, then $M$ is the uniform
  matroid $U_{k,k}$; this can be detected from $\mathcal{G}(M)$, which
  is $k![11\ldots 1]$.  It follows from Theorem \ref{thm:higgs} that
  $Q\setminus (E\cup a)$ is also $U_{k,k}$, so
  $\mathcal{G}(Q)=\mathcal{G}(M)$.\qed
\end{proof}

The case of $Q\setminus(E\cup a)$ with $m=1$, which is isomorphic
to the Higgs lift of $M$, was treated in \cite[Proposition
4.2]{catdata} using rank sequences; the expression above using the
$\gamma$-basis is new.

We extend Theorem \ref{thm:config} to the variants of the free
$m$-cone.  We use the following observations about cyclic flats in a
deletion $N\setminus X$ of a matroid $N$.  If
$F\in\mathcal{Z}(N\setminus X)$, then $\cl_N(F)\in\mathcal{Z}(N)$.
The flats of $N\setminus X$ are the sets $F-X$ as $F$ ranges over the
flats of $N$.  So the cyclic flats of $N\setminus X$ are the sets
$F-X$, with $F\in\mathcal{Z}(N)$, that are cyclic.
  
\begin{theorem}\label{thm:config tipless}
  Let $M$ be a loopless matroid and let $Q=Q_m(M)$.
  \begin{enumerate}
  \item[\emph{(1)}] If $m>1$, then the configuration of $Q\setminus a$
    determines $M$.
  \item[\emph{(2)}] If $m>1$, then $m$ and the configuration of
    $Q\setminus E$ determine $M$.
  \item[\emph{(3)}] If $m>2$, then $m$ and the configuration of
    $Q\setminus(E\cup a)$ determine $M$.
  \end{enumerate}
\end{theorem}

\begin{proof}
  Each line of $Q$ that contains $a$ has at least four points, so,
  using the observations above, the cyclic flats of $Q\setminus a$ are
  those of $M$ along with the sets $q(F)-a$ as $F$ ranges over the
  nonempty flats of $M$.  Let
  $\mathcal{L} =\{L-a\,:\,L \text{ is a line of } Q \text{ with } a\in
  L\}$.  From the proof of Theorem \ref{thm:config}, we see that
  $\mathcal{L}$ is the largest set of cyclic lines of $Q\setminus a$
  that satisfies properties (L1)--(L3) in that proof (using $E(Q)-a$
  in (L2)), and we can deduce, from the configuration of
  $Q\setminus a$, which of its elements come from lines in
  $\mathcal{L}$.  The cyclic flats of $Q$ that contain $a$ correspond
  to the elements $F$ in the configuration of $Q\setminus a$ where
  $L\leq F$ for at least one $L\in\mathcal{L}$, so we get the
  configuration of $Q$ by increasing the size assigned to such
  elements by $1$.  So, from the configuration of $Q\setminus a$, we
  get that of $Q$, from which we get $M$ by Theorem \ref{thm:config}.

  The cyclic flats of $Q\setminus E$ are $\varnothing$ and the sets
  $q(F)-F$ as $F$ ranges over the nonempty flats of $M$.  Increase the
  size assigned to each element $F$ in the configuration of
  $Q\setminus E$ by $$\sum\limits_{\substack{L \text{
        in the configuration,} \\
      L\leq F, \text{ and } r(L)=2}}\frac{|L|-1}{m}.$$ This gives the
  data that was used in the proof of Theorem \ref{thm:config} to get
  $M$.

  Similarly, when $m>2$, the cyclic flats of $Q\setminus (E\cup a)$
  are precisely $\varnothing$ and the sets $q(F)-(F\cup a)$ as $F$
  ranges over the nonempty flats of $M$.  For each element of positive
  rank in the configuration of $Q\setminus (E\cup a)$, increase its
  assigned size by $1$; this gives the configuration of
  $Q\setminus E$, from which we just showed that we get $M$.\qed
\end{proof}

The matroids $M_1$ and $M_2$ in Figure \ref{fig:continuing examples}
show that the inequalities $m>1$ for (1) and (2), and $m>2$ for (3),
are necessary.

\section{Size-Rank-Coloop Data}\label{section:size rank coloop}

In contrast to Theorem \ref{thm:G-inv}, the Tutte polynomial of $M$
does not determine that of $Q$: the matroids $N_1$ and $N_2$ in Figure
\ref{fig:counterex} have the same Tutte polynomial, but their free
$1$-cones do not.  (The same counterexample settles this question also
for the characteristic polynomial, which is
$\chi(M;x) = (-1)^{r(M)}T(M;1-x,0)$.)  We introduce a matroid
invariant that determines the Tutte polynomial of $Q$.

The Tutte polynomial $T(M;x,y)$ of $M$ is equivalent to the
\emph{size-rank data} of $M$, which is the multiset of pairs
$(|S|,r(S))$ for all $S\subseteq E(M)$.  The \emph{size-rank-coloop
  data} of $M$ is the multiset of triples $(|S|,r(S),c(S))$, for all
$S\subseteq E(M)$, where $c(S)$ is the number of coloops of $M|S$. We
show that we can compute $T(Q;x,y)$ using this stronger invariant.

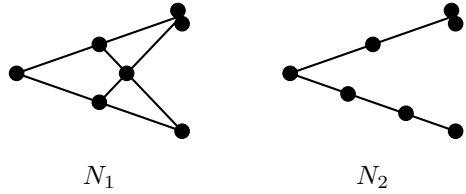
\begin{figure}[t]
  \centering
  \begin{tikzpicture}[scale=1.1]
    \draw[thick](0,0.5)--(2,-0.2);
    \draw[thick](0,0.5)--(2,1.2);
    \draw[thick](1,0.15)--(2,1.2);
    \draw[thick](1,0.85)--(2,-0.2);
    \filldraw (1,0.85) circle  (2.5pt);
    \filldraw (2,1.1) circle  (2.5pt);
    \filldraw (1.95,1.26) circle  (2.5pt);
    \filldraw (2,-0.2) circle  (2.5pt);
    \filldraw (1,0.15) circle  (2.5pt);
    \filldraw (1.33,0.5) circle  (2.5pt);
    \filldraw (0,0.5) circle  (2.5pt);
    \node at (1,-0.75) {$N_1$};
  \end{tikzpicture}
  \hspace{30pt}
  \begin{tikzpicture}[scale=1.1]
    \draw[thick](0,0.5)--(2,-0.2);
    \draw[thick](0,0.5)--(2,1.2);
    \filldraw (1,0.85) circle  (2.5pt);
    \filldraw (2,1.1) circle  (2.5pt);
    \filldraw (1.95,1.26) circle  (2.5pt);
    \filldraw (2,-0.2) circle  (2.5pt);
    \filldraw (0.7,0.251) circle  (2.5pt);
    \filldraw (1.4,0.017) circle  (2.5pt);
    \filldraw (0,0.5) circle  (2.5pt);
    \node at (1,-0.75) {$N_2$};
  \end{tikzpicture}
  \caption{The matroids above have the same Tutte polynomial, but
    $\mathcal{G}(N_1)\ne \mathcal{G}(N_2)$. Their free $1$-cones have
    different Tutte polynomials.}\label{fig:counterex}
\end{figure}

\begin{theorem}\label{thm:size rank coloop}
  Let $Q=Q_m(M)$ for some $m\geq 1$. The Tutte polynomial of $Q$ can
  be computed from the size-rank-coloop data of $M$.
\end{theorem}

\begin{proof}
  We refer to the following information as the \emph{type} of a subset
  $S'$ of $E(Q)$:
  \begin{itemize}
  \item[(1)] the set $S=p(S')$,
  \item[(2)] $r(S)$,
  \item[(3)] whether $a\in S'$,
  \item[(4)] for each $e\in S$, (a) the size $|S'\cap T_e|$, (b)
    whether $e \in S'$, and (c) whether $e$ is a coloop of $M|S$.
  \end{itemize}
  We call $S$ the underlying set of the type.  Below we show that from
  the type of $S'$, we can deduce $|S'|$ and $r(S')$, and so the
  contribution of $S'$ to $T(Q;x,y)$.  The proof is then completed by
  showing that for each $S\subseteq E$ and each possible type of
  subsets of $E(Q)$ with underlying set $S$, we can find the number of
  subsets of $E(Q)$ having that type from the triple
  $(|S|,r(S),c(S))$.

  Let $s=|S\cap S'|+\sum_{e\in S} |S'\cap T_e|$.  If $a\not\in S'$,
  then $|S'|=s$, otherwise $|S'|=s+1$.  Thus, we get $|S'|$ from the
  type of $S'$.

  Now $r(S')=r(S)$ if $a\not\in\cl(S')$, otherwise $r(S')=r(S)+1$, by
  Corollary \ref{cor:rank}, so to show that we get $r(S')$ from the
  type of $S'$, it suffices to show that the type of $S'$ determines
  whether $a\in\cl(S')$.  We have $a\in\cl(S')$ if either $a\in S'$ or
  $|S'\cap (T_e\cup e)|>1$ for some $e\in S$, so now assume
  $a\not\in S'$ and $|S'\cap (T_e\cup e)|=1$ for all $e\in S$.

  With those assumptions, we claim that $a\in \cl(S')$ if and only if
  $|S'\cap T_e|>0$ for some $e\in S$ that is not a coloop in $M|S$.
  First assume that $a\in \cl(S')$.  Now $\cl(S')=q(F)$ for some flat
  $F$ of $M$ by Lemma \ref{lemma:flats}. Let $B\subseteq S'$ be a
  basis of $Q|q(F)$.  We have $|p(B)|=|B|=r(F)+1$ and $r(p(B))=r(F)$,
  so $p(B)$ contains a circuit $C$.  Since $C\not\subseteq B$, there
  is an $e\in C-S'$.  So $e\in S$, it is not a coloop in $M|S$, and
  $|S'\cap T_e|>0$.  Now assume that $a\not\in \cl(S')$.  If $e\in S$
  and $e'\in S'\cap T_e$, then $e'$ is a coloop of $Q|S'$ by Lemma
  \ref{lemma:coloops}.  Since $r(S)=r(S')$ and $r(S-e)=r(S'-e')$, it
  follows that $e$ is a coloop of $M|S$.  This completes the proof
  that we get $r(S')$ from the type of $S'$.

  We now show that for each set $S\subseteq E$ and possible type with
  underlying set $S$, we get the number of subsets of $E(Q)$ of that
  type from the triple $(|S|,r(S),c(S))$.  As just shown, $r(S')=r(S)$
  if and only if (i) $a\not\in S'$, (ii) $|S'\cap (T_e\cup e)|=1$ for
  all $e\in S$, and (iii) each $e\in S$ with $|S'\cap T_e|=1$ is a
  coloop of $M|S$; otherwise $r(S')=r(S)+1$.  For any positive integer
  $h$, the number of sets $S'$ with $p(S')=S$ and $|S'|=h$ depends
  just on $m$ and $|S|$, so the proof is completed by observing that
  the number of such sets $S'$ that satisfy conditions (i)--(iii) is
  determined by $|S|$ and $c(S)$.\qed
\end{proof}

The free $1$-cones of $N_1$ and $N_2$ in Figure \ref{fig:counterex}
have different Tutte polynomials, so $N_1$ and $N_2$ have different
size-rank-coloop data. Indeed, $N_1$ has $20$ sets of size $4$, rank
$3$, with one coloop, while $N_2$ has $18$.

The proof of Theorem \ref{thm:size rank coloop} applies, with only
minor changes in which sets $S'$ are of concern, to prove the following
extension to the variants of the free $m$-cone introduced in Section
\ref{section:variants}.

\begin{corollary}
  Let $Q=Q_m(M)$ for some $m\geq 1$. The Tutte polynomial of
  $Q\setminus a$, of $Q\setminus E$, and of $Q\setminus(E\cup a)$
  can be determined from the size-rank-coloop data of $M$.
\end{corollary}

We now prove that $\mathcal{G}(M)$ gives the size-rank-coloop data of
$M$.

\begin{theorem}
  Let $M=(E,r)$ be a matroid with $|E|=n$. The size-rank-coloop data
  of $M$ can be computed from $\mathcal{G}(M)$.
\end{theorem}

\begin{proof}
  For a triple of nonnegative integers $(s,t,c)$ with $s\geq t\geq c$,
  let $F(s,t,c)$ be the set of subsets $X$ of $E$ with $|X|=s$ and
  $r(X)=t$ for which $M|X$ has $c$ coloops, and let $G(s,t,c)$ be the
  set of permutations $\pi$ of $E$ for which $[\underline{r}(\pi)]$
  has exactly $t$ ones in the first $s$ places, with each of the last
  $c$ of those being one.  More elements than those corresponding to
  the last string of ones may be coloops among the first $s$ elements,
  so if $\pi\in G(s,t,c)$, then
  $\{\pi(i)\,:\,i\in[s] \}\in \bigcup_{c'\geq c} F(s,t,c')$.  A set
  $X \in F(s,t,c')$ gives rise to permutations $\pi\in G(s,t,c)$ by
  picking $c$ of the coloops, taking a permutation of $X$ with those
  $c$ coloops last, and placing the elements of $E-X$ after those $c$
  coloops.  So,
  $$\sum_{c'\geq c}|F(s,t,c')|\binom{c'}{c}c!(s-c)!(n-s)! = |G(s,t,c)|.$$
  We get $|G(s,t,c)|$ from $\mathcal{G}(M)$.  For a fixed $s$ and $t$,
  these equalities form an upper triangular system of equations in
  $|F(s,t,c)|$, so, as needed, we can find the terms $|F(s,t,c)|$.\qed
\end{proof}

The results above show that the size-rank-coloop data is strictly
stronger than the Tutte polynomial, and the $\mathcal{G}$-invariant is
at least as strong as the size-rank-coloop data.  It seems most likely
that the $\mathcal{G}$-invariant is strictly stronger than the
size-rank-coloop data, but we do not currently have pairs of matroids
with the same size-rank-coloop data and different
$\mathcal{G}$-invariants.  In summary, we have the invariants below
listed from weakest to strongest, where (2) is strictly stronger than
(1), and (4) is strictly stronger than (3), but we do not yet know
about (3) and (2):
\begin{enumerate}
\item[(1)] the Tutte polynomial,
\item[(2)] the size-rank-coloop data,
\item[(3)] the $\mathcal{G}$-invariant, and
\item[(4)] the configuration.
\end{enumerate}

\vspace{5pt}

\begin{center}
 \textsc{Acknowledgments}
\end{center}

\vspace{3pt}

We thank Joseph Kung for valuable observations and comments on the
topics of this paper.  We thank both referees for their careful
reading of the manuscript, their valuable comments, and for pointing
out a gap in the original proof of Theorem \ref{thm:config}.

\bibliographystyle{alpha}

\end{document}